\documentclass[10pt]{article}

\usepackage{amsmath,amsthm,amssymb}
\usepackage{caption,graphicx,subfig}
\usepackage[text={6.5in,9.25in},centering]{geometry}
\usepackage{paralist}
\usepackage[sort,numbers]{natbib}

\setcaptionmargin{0.25in}

\makeatletter\@addtoreset{equation}{section}\makeatother

\newtheorem{thm}{Theorem}
\newtheorem{lem}{Lemma}[section]

\theoremstyle{definition}

\newcommand{\drm}{\mathrm{d}}
\newcommand{\R}{\mathbb{R}}

\begin{document}

\title{Exact minimum speed of traveling waves in a Keller--Segel model}

\author{
Jason J. Bramburger\\
Department of Mathematics and Statistics\\
University of Victoria\\
Victoria, BC, V8P 5C2, Canada
}

\date{}
\maketitle

\begin{abstract}
In this paper we present a Keller--Segel model with logistic growth dynamics arising in the study of chemotactic pattern formation. We prove the existence of a minimum wave speed for which the model exhibits nonnegative travelling wave solutions at all speeds above this value and none below. The exact value of the minimum wave speed is given for all biologically relevant parameter values. These results strengthen recent results where non-sharp upper and lower bounds on the minimum wave speed were derived in a restricted parameter regime.
\end{abstract}


\section{Introduction}

Chemotaxis can be responsible for fascinating pattern-forming behaviour in bacterial populations, even in homogeneous landscapes \cite{Adler,Adler2,Gharasoo,Murray}. To better understand such complex behaviour Keller and Segel proposed the class of relatively simple mathematical models to accurately describe chemotaxis, which now bear their name \cite{Keller,Keller2}. In this letter we will consider such a Keller--Segel model for chemotaxis of the form 
\begin{equation}\label{KS}
	\begin{split}
		u_t &= u_{xx} - (u\chi(v)v_x)_x + \mu u(1-u), \\
		v_t &= Dv_{xx} +  \beta v - u.
	\end{split}	
\end{equation}
The real-valued quantities $u(x,t) \geq 0$ and $v(x,t) \geq 0$ denote the cell density and the concentration of the chemical signal, respectively, while subscripts denote partial differentiation with respect to the spatial variable $x \in \mathbb{R}$ and the temporal variable $t \geq 0$. The parameter $\mu > 0$ represents the rate of logistic cell growth, while $\chi(v) \geq 0$, the chemotactic sensitivity, measures the effect the chemical substance that is produced over time has on the cell population. In the second equation of \eqref{KS} the constant $\beta > 0$ describes the chemical growth rate, and the rate at which the chemical is consumed by the cells is fixed to be 1 to simplify the notation. Finally, the diffusion coefficient $D \geq 0$ is taken to represent the ratio of the diffusion coefficients of the cell and the chemical attractant.   

Among the simplest nontrivial solutions to biologically relevant spatially extended systems are traveling waves. Such solutions are steady spatial patterns that linearly propagate over the spatial medium at a constant speed, and in the context of \eqref{KS} represent the densities spreading or receding over the spatial medium. In the present case of system \eqref{KS}, the logistic nonlinearity in the first equation induces monostable nonlinear dynamics, the presence of which often leads to a nonzero minimum wave speed, denoted $c_* $, such that traveling waves exist at all faster speeds and no slower ones \cite{Saarloos}. The importance of the minimum wave speed goes beyond just the existence of traveling wave solutions since it has been proven in a number of monostable systems that compactly supported initial conditions will spread with asymptotic speed given by $c_*$ \cite{Aaronson, Gray, KPP, Saarloos, Stokes}. Compactly supported initial conditions are of course more biologically relevant since they represent initial population densities localized to a region of space, and the analysis of minimum wave speeds to traveling wave solutions provides an avenue for understanding the spread of these populations.   

In the context of \eqref{KS}, a traveling wave moving at speed $c > 0$ takes the form $(u(x,t),v(x,t)) = (U(x-ct),V(x-ct))$, where the nonnegative profiles $U$ and $V$ satisfy $(U(-\infty),V(-\infty)) = (1,\beta^{-1})$ and $(U(\infty),V(\infty)) = (0,0)$. Such a solution describes a connection between the homogeneous steady-states $(u,v) = (0,0)$ and $(u,v) = (1,\beta^{-1})$ for which the latter is advancing into the former. The existence of a minimum wave speed $c_*$ for system \eqref{KS} with $D = 0$ was recently confirmed in \cite{Li}, meaning that traveling wave solutions of \eqref{KS} exist for all speeds $c \geq c_*$. Under the assumption that $0 \leq \chi(v) \leq \mu$ for all $v \in [0,\beta^{-1}]$ the authors provided the bound    
\begin{equation}\label{LiBnd}
	\begin{split}
		2\sqrt{\mu}\leq c_* \leq \max_{0 \leq v \leq \beta^{-1}}\bigg\{2\sqrt{\mu},\sqrt{\frac{\beta}{\mu}\chi(v)}\bigg\}.
	\end{split}
\end{equation}
Notice that in certain parameter regimes the gap between upper and lower bounds can be very large, giving little information about the minimum wave speed. For example, previous studies have considering Keller--Segel models similar to \eqref{KS} with $\chi(v)$ a constant function \cite{Landman,Nadin,Salako,Salako2}, which in our case implies that considering large $\beta$ can produce bounds on $c_*$ that give little information aside from the existence of a minimum wave speed. Using the computational techniques for bounding the minimum wave speed put forth in \cite{Bramburger}, it was noticed that in all tested cases the upper bound is independent of the choice of the function $\chi$. Informed by these numerical results, in this letter we refine the results of \cite{Li} to produce the exact value of $c_*$ with the following theorem.

\begin{thm}\label{thm:Main1}
	Fix $D = 0$. Then, for all $\beta,\mu > 0$ and functions $\chi$ satisfying $0 \leq \chi(v) \leq \mu$ for all $v \in [0,\beta^{-1}]$, the system \eqref{KS} has a unique nonnegative traveling wave solution $(u(x,t),v(x,t)) = (U(x-ct),V(x-ct))$ satisfying $(U(-\infty),V(-\infty)) = (1,\beta^{-1})$ and $(U(\infty),V(\infty)) = (0,0)$ if, and only if, $c \geq 2\sqrt{\mu}$. Furthermore, this traveling wave satisfies 
	\begin{equation}\label{Ordering}
		0 \leq U(x-ct) \leq \beta V(x-ct) \leq 1,
	\end{equation}
	for all $x \in \mathbb{R}$ and $t \geq 0$ and the profiles $U,V$ are monotone decreasing.
\end{thm}

Theorem~\ref{thm:Main1} improves the upper bound in \eqref{LiBnd} to show that it is {\em independent} of the function $\chi$, thus coinciding with the lower bound. The proof is left to Section~\ref{sec:Proof} where we convert the problem of finding traveling waves to determining the existence of heteroclinic orbits to a dynamical system in the spatial variable $x-ct$. In this spatial dynamical system we construct an appropriate trapping region for $c \geq 2\sqrt{\mu}$ that guarantees the existence of the desired heteroclinic orbit, which is constructed in such a way to give the ordering \eqref{Ordering} and monotonicity results as well. The methods of proving Theorem~\ref{thm:Main1} can be extended to understand traveling wave solutions to \eqref{KS} with $D > 0$, resulting in the following theorem whose proof is left to Section~\ref{sec:Proof2}

\begin{thm}\label{thm:Main2}
	For all $D,\beta,\mu > 0$ and functions $\chi$ satisfying $0 \leq \chi(v) \leq \mu$ for all $v \in [0,\beta^{-1}]$, the system \eqref{KS} has a unique nonnegative traveling wave solution $(u(x,t),v(x,t)) = (U(x-ct),V(x-ct))$ satisfying $(U(-\infty),V(-\infty)) = (1,\beta^{-1})$ and $(U(\infty),V(\infty)) = (0,0)$ if, and only if, $c \geq 2\max\{\sqrt{\mu},\sqrt{D\beta}\}$. Furthermore, this traveling wave satisfies the ordering \eqref{Ordering} for all $x \in \mathbb{R}$ and $t \geq 0$ and the profiles $U,V$ are monotone decreasing.
\end{thm}


\section{Proof of Theorem~\ref{thm:Main1}}\label{sec:Proof} 

Throughout this section we assume that $D = 0$, $\beta,\mu > 0$ and $\chi$ satisfies $0 \leq \chi(v) \leq \mu$ for all $v \in [0,\beta^{-1}]$. Then, traveling wave solutions of \eqref{KS} take the form $u(x,t) = U(\xi)$ and $v(x,t) = V(\xi)$, where $\xi := x - ct$ and the parameter $c \geq 0$ is the wave speed. With this ansatz $U(\xi)$ and $V(\xi)$ must solve 
\begin{equation}
	\begin{split}
		-c\frac{\drm U}{\drm \xi} &=  \frac{\drm^2U}{\drm \xi^2} - \frac{\drm}{\drm \xi}\bigg(U\chi(V)\frac{\drm V}{\drm \xi}\bigg) + \mu U (1-U), \\
		-c\frac{\drm V}{\drm \xi} &=  \beta V - U,
	\end{split}	
\end{equation}
where we recall that the second equation has no second order derivative since $D = 0$. Introducing $W = cU + \frac{\drm U}{\drm \xi} - \bigg(U\chi(V)\frac{\drm V}{\drm \xi}\bigg)$ results in the three-dimensional ordinary differential equation
\begin{equation}\label{KSODE}
	\begin{split}
		\dot{U} &= -cU + \frac{U\chi(V)(U-\beta V)}{c} + W, \\
		\dot{V} &= \frac{1}{c}(U - \beta V), \\
		\dot{W} &= \mu U(U-1),
	\end{split}
\end{equation}
where the dot represents differentiation with respect to the independent variable $\xi$. We seek a heteroclinic orbit which asymptotically connects the co-operation state $(1,\beta^{-1},c)$ to the origin $(0,0,0)$, representing a traveling invasion wave in the full partial differential equation (\ref{KS}) moving with constant speed $c > 0$. We present the following lemma from \cite{Li} regarding the stability of the origin of \eqref{KSODE}.

\begin{lem}[\cite{Li}]\label{lem:Origin}
	The origin of \eqref{KSODE} is locally asymptotically stable for all $c > 0$. In particular, when $0 < c < 2\sqrt{\mu}$ the linearization of \eqref{KSODE} about the origin has a negative real eigenvalue and a pair of complex-conjugate eigenvalues with negative real part, and when $c \geq 2\sqrt{\mu}$ all eigenvalues are real and negative. 
\end{lem}

From Lemma~\ref{lem:Origin} we can see that for $0 < c < 2\sqrt{\mu}$ a heteroclinic orbit (if it exists) would necessarily have negative $U$ and/or $V$ components. This of course violates the condition that we are seeking nonnegative traveling waves and therefore we have the explicit lower bound $c_* \geq 2\sqrt{\mu}$, proving the `only if' portion of Theorem~\ref{thm:Main1}.

Now, for a fixed $c > 0$ we define the region 
\begin{equation}
	\mathcal{R} = \{(U,V,W):\ 0\leq U \leq \beta V \leq 1, 0 \leq W \leq cU\}.	
\end{equation}
We note that in this region the $U,V,$ and $W$ components are monotone decreasing, and therefore it is easily shown that any trajectory that enters and never leaves $\mathcal{R}$ must converge to the trivial equilibrium. This leads to the following lemma.

\begin{lem}\label{lem:Trap} 
	For each $c > 0$ the unstable manifold of the equilibrium $(1,\beta^{-1},c)$ of \eqref{KSODE} is one-dimensional and its tangent vector at this equilibrium points to the interior of $\mathcal{R}$. Furthermore, the unstable manifold can only leave $\mathcal{R}$ by crossing $W = 0$.
\end{lem}	

\begin{proof}
	The first statement of the proposition detailing that the unstable manifold of $(1,\beta^{-1},c)$ is one-dimensional and that its tangent vector at this equilibrium points to the interior of $\mathcal{R}$ was proven in \cite{Li}. Therefore, we are only left to prove the second statement. We now examine the dynamics on the faces making up the boundary of $\mathcal{R}$ to show that a trajectory which enters $\mathcal{R}$ can only leave by crossing the face $W = 0$.   
	
	On the face $U = 0$ we have $\dot{U} = W$, which is nonnegative for $0 \leq W \leq cU$, and hence trajectories cannot leave $\mathcal{R}$ through this face. Similarly, on the face $V = 1/\beta$ we have $\dot{V} = \frac{1}{c}(U - 1)$, which is nonpositive since $U \leq 1$ in $\mathcal{R}$, and so solutions cannot leave through this face either. On the face $U = \beta V$ we have 
\begin{equation}
	\dot{U} - \beta\dot{V} = W - cU,
\end{equation}
which is nonpositive when $W \leq cU$. Hence, solutions of (\ref{KS}) cannot leave $\mathcal{R}$ through the face $U = \beta V$. Similarly, on the face $W = cU$ we have
\begin{equation}
	\dot{W} - c\dot{U} = \mu U(U-1) - U(U - \beta V)\chi(V) = (\mu - \chi(V))U(U - \beta V) + \mu U(\beta V - 1),
\end{equation}
and since in $\mathcal{R}$ we have $U \leq \beta V$, $\beta V \leq 1$ and we have assumed that $0 \leq \chi(V) \leq \mu$ for all $V \in [0,\beta^{-1}]$, it follows that $\dot{W} - c\dot{U}$ is nonpositive when $W = cU$. Therefore, trajectories cannot leave $\mathcal{R}$ through the face $W = cU$. Hence, the only way by which trajectories may leave $\mathcal{R}$ is by crossing $W= 0$. \qed
\end{proof} 

Using Lemma~\ref{lem:Trap} we find that to guarantee that the unstable manifold of $(1,\beta^{-1},c)$ never leaves $\mathcal{R}$ we require a lower boundary in this set which cannot be crossed by solutions of (\ref{KS}). That is, we seek to construct a surface $W = N(U,V)$ which guarantees that solutions of the differential equation cannot leave through the bottom of $\mathcal{R}$, for appropriate $c > 0$. Precisely, the function $N$ must satisfy: 
\begin{enumerate}
	\item $\frac{\drm}{\drm t}(W - N(U,V)) \geq 0$ along trajectories of (\ref{KSODE}) on $W = N(U,V)$ , 
	\item $0 \leq N(U,V) \leq cU$ for all $0 \leq U \leq \beta V \leq 1$.
	\item $N(0,0) = 0$.
\end{enumerate}
Condition(1) implies that trajectories can only cross the surface $W = N(U,V)$ from $W > N(U,V)$ to $W < N(U,V)$. Condition (2) guarantees that the surface $N$ lies inside of the region $\mathcal{R}$ and Condition (3) guarantees that the can be reached by trajectories bounded below by the surface $N$. Together these conditions mean that if such a surface can be found, then trajectories that start above $N$ in $\mathcal{R}$ cannot leave $\mathcal{R}$, and therefore must converge to the origin $(0,0,0)$ as $\xi \to \infty$. This leads to the following result, which together with our previous discussion proves Theorem~\ref{thm:Main1}.

\begin{thm}\label{thm:Het1} 
	For all $c \geq 2\sqrt{\mu}$ system (\ref{KSODE}) has a heteroclinic orbit from $(1,\beta^{-1},c)$ to $(0,0,0)$ which remains in $\mathcal{R}$ for all $\xi \in \R$.
\end{thm}

\begin{proof}
	Based on the preceding discussion, we simply require finding a surface $W = N(U,V)$ which satisfies the conditions listed above. Let us take $N(U,V) = \eta U$, for a yet-to-be specified $\eta > 0$, and note that Condition (3) is immediately satisfied. Then, for $N$ to satisfy Condition (1) requires finding $c,\eta > 0$ such that 
	\begin{equation}
		0 \leq \dot{W} - N_U\dot{U} - N_V\dot{V} = \mu U(U-1) - \eta\bigg[-cU + \frac{U(U-\beta V)\chi(V)}{c} + \eta U\bigg].
	\end{equation}
	Rearranging this expression now gives 
	\begin{equation}\label{etaQuad}
		\eta^2 +\bigg[ \frac{(U-\beta V)\chi(V)}{c} - c\bigg]\eta + \mu(1-U) \leq 0,
	\end{equation}
	which is a quadratic expression in $\eta > 0$ with positive leading term. Therefore, to satisfy \eqref{etaQuad} we require that the discriminant of the left-hand-side is nonnegative, i.e.
	\begin{equation}\label{Discriminant}
		\bigg[ \frac{(U-\beta V)\chi(V)}{c} - c\bigg]^2 \geq 4\mu(1-U),	
	\end{equation}
	For all $0 \leq U \leq \beta V \leq 1$ we have that $\frac{(U-\beta V)\chi(V)}{c} \leq 0$, so that the left-hand-side of the inequality (\ref{Discriminant}) can be minimized over all relevant $0 \leq U \leq \beta V \leq 1$ to find that we equivalently require 
	\begin{equation}
		c^2 \geq 4\mu(1-U) \implies c^2 \geq 4\mu,
	\end{equation} 
	giving that $c \geq 2\sqrt{\mu}$. 
	
	Now that we have shown that there exists values of $\eta \in \mathbb{R}$ that satisfy \eqref{etaQuad} for all $c \geq 2\sqrt{\mu}$, we now show that such an $\eta$ can be chosen to guarantee that Condition (2) for $N(U,V) = \eta U$ is satisfied. That is, that an $\eta$ which satisfies \eqref{etaQuad} also satisfies $0 < \eta \leq c$. First, since the discriminant \eqref{Discriminant} is nonnegative for all $c \geq 2\sqrt{\mu}$, the quadratic expression in $\eta$ has two roots in $\eta$ for all $0 \leq U \leq \beta V\leq 1$ and the expression \eqref{etaQuad} is negative between these roots. 
	
	Let us take $\eta = \eta_\mathrm{crit}$ to be the $\eta$-value of the critical point of \eqref{etaQuad}. The location of $\eta_\mathrm{crit}$ is given by
	 \begin{equation}\label{etaCrit}
	 	0 < \eta_\mathrm{crit} = \frac{c}{2} - \frac{\chi(V)(U - \beta V)}{2c} \leq \frac{c}{2} + \frac{\chi(V)}{2c} \leq \frac{c}{2} + \frac{\sqrt{\mu}}{4} \leq \frac{c}{2} + \frac{c}{8} = \frac{5c}{8} < c, 
	 \end{equation}
	 where we have used the fact that $c \geq 2\sqrt{\mu}$ and $\chi(V) \leq \mu$ for all $V \in [0,\beta^{-1}]$. Hence, taking $\eta = 5c/8$ gives that the surface satisfies Conditions (1) and (2) for every $c \geq 2\sqrt{\mu}$, thus giving that a heteroclinic orbit of (\ref{KSODE}) that remains in $\mathcal{R}$ for all $\xi\in\R$ exists for all $c \geq 2\sqrt{\mu}$. This concludes the proof. \qed
\end{proof} 

\section{Proof of Theorem~\ref{thm:Main2}}\label{sec:Proof2} 

Due to the relative simplicity of the case when $D = 0$ compared to the work in \cite{Li}, it is now a fairly straightforward task to extend the results to all $D > 0$. Introducing again $u(x,t) = U(\xi)$ and $v(x,t) = V(\xi)$, where $\xi = x - ct$, into \eqref{KS} with $D > 0$ implies that $U(\xi)$ and $V(\xi)$ must solve 
\begin{equation}
	\begin{split}
		-c\frac{\drm U}{\drm \xi} &=  \frac{\drm^2U}{\drm \xi^2} - \frac{\drm}{\drm \xi}\bigg(U\chi(V)\frac{\drm V}{\drm \xi}\bigg) + \mu U (1-U), \\
		-c\frac{\drm V}{\drm \xi} &= D\frac{\drm^2V}{\drm \xi^2} +  \beta V - U,
	\end{split}	
\end{equation}
We again introduce $W = cU + \frac{\drm U}{\drm \xi} - \bigg(U\chi(V)\frac{\drm V}{\drm \xi}\bigg)$ and now set $Y = \frac{\drm V}{\drm \xi}$, to arrive at the four-dimensional ordinary differential equation
\begin{equation}\label{KSODE2}
	\begin{split}
		\dot{U} &= -cU + U\chi(V)Y + W, \\
		\dot{V} &= Y, \\
		\dot{Y} &= -cD^{-1}Y + D^{-1}(U - \beta V), \\
		\dot{W} &= \mu U(U-1).
	\end{split}
\end{equation}
In the system \eqref{KSODE2} the traveling waves of interest are heteroclinic orbits connecting $(U,V,Y,W) = (1,\beta^{-1},0,c)$ to the origin $(U,V,Y,W) = (0,0,0,0)$. We now state the following lemma, analogous to Lemma~\ref{lem:Origin}, without proof since it can be calculated directly.

\begin{lem}\label{lem:Origin2}
	The origin of \eqref{KSODE2} is locally asymptotically stable for all $c > 0$. In particular, when $0 < c < 2\max\{\sqrt{\mu},\sqrt{D\beta}\}$ the linearization of \eqref{KSODE} about the origin has at least one pair of complex-conjugate eigenvalues with negative real part, and when $c \geq 2\max\{\sqrt{\mu},\sqrt{D\beta}\}$ all eigenvalues are real and negative.	
\end{lem}

Again we see that to have nonnegative traveling waves we require $c \geq 2\max\{\sqrt{\mu},\sqrt{D\beta}\}$, making up the `only if' portion of Theorem~\ref{thm:Main2}. Now, let us restrict our attention to the case that $c \geq 2\max\{\sqrt{\mu},\sqrt{D\beta}\}$ and define 
\begin{equation}\label{rho}
	\rho := (c - \sqrt{c^2 - 4D\beta})/2D\beta > 0.
\end{equation} 
This leads to the definition of the region
\begin{equation}
	\mathcal{S} = \{(U,V,Y,W):\ 0 \leq U\leq \beta V\leq 1,\ 0\leq W \leq cU,\ \rho(U-\beta V) \leq Y \leq 0\}.
\end{equation}
Notice that if a heteroclinic orbit remains in $\mathcal{S}$ for all $\xi \in \mathbb{R}$, we again obtain the desired ordering of the solutions \eqref{Ordering} and the monotonicity of the $U,V,$ and $W$ components. This leads to the following lemma.

\begin{lem}\label{lem:Trap2} 
	For each $c > 0$ the unstable manifold of the equilibrium $(1,\beta^{-1},0,c)$ of \eqref{KSODE2} is one-dimensional and its tangent vector at this equilibrium points to the interior of $\mathcal{S}$. Furthermore, the unstable manifold can only leave $\mathcal{S}$ by crossing $W = 0$.
\end{lem}	

\begin{proof}
	Much of the proof is the same as the proof of Lemma~\ref{lem:Trap}, so we focus exclusively on the addition of the $Y$ variable which distinguishes the region $\mathcal{S}$ from $\mathcal{R}$. We first note that if $Y = 0$ we have 
	\begin{equation}
		\dot{Y} = D^{-1}(U - \beta V), 
	\end{equation}
	which is nonnegative for $0 \leq U \leq \beta V \leq 1$. Hence, the unstable manifold of $(1,\beta^{-1},0,c)$ cannot leave $\mathcal{S}$ through the face $Y = 0$. Then, on the face $Y = \rho(U-\beta V)$ we have 
	\begin{equation}\label{YTrap}
		\dot{Y} - \rho(\dot{U} - \beta\dot{V}) = \underbrace{\bigg(\beta\rho +\frac{1}{D\rho} - \frac{c}{D}\bigg)}_{=0}Y - \rho(W - cU + U\chi(V)Y)
	\end{equation}
	where the first term on the right of the equality vanishes from our choice of $\rho > 0$ in \eqref{rho}. When $0 \leq W \leq cU$ and $Y \leq 0$ the second term on the right of the equality in \eqref{YTrap} is nonnegative, implying that the unstable manifold of $(1,\beta^{-1},0,c)$ cannot leave $\mathcal{S}$ through the face $Y = \rho(U-\beta V)$. This concludes the proof. \qed   
\end{proof} 

From Lemmas~\ref{lem:Origin2} and \ref{lem:Trap2} we see that the situation is now completely analogous to the previous section when $D = 0$. This leads to the final result of the section, which together with the previous lemmas proves Theorem~\ref{thm:Main2}.

\begin{thm}\label{thm:Het2} 
	For all $c \geq 2\max\{\sqrt{\mu},\sqrt{D\beta}\}$ system (\ref{KSODE2}) has a heteroclinic orbit from $(1,\beta^{-1},0,c)$ to $(0,0,0,0)$ which remains in $\mathcal{S}$ for all $\xi \in \R$.
\end{thm}

\begin{proof}
	The proof is almost identical to that of Theorem~\ref{thm:Het1} since we only make use of the $\dot{U}$ and $\dot{W}$ equations. The only major difference is the handling of \eqref{etaCrit}. That is, in the present case we get
	\begin{equation}\label{etaCrit2}
		0 \leq \eta_\mathrm{crit} =  \frac{c}{2} - \frac{\chi(V)Y}{2} \leq  \frac{c}{2} - \frac{\rho\chi(V)(U - \beta V)}{2},   
	\end{equation} 
	using the fact that $\rho(U - \beta V) \leq Y$ in $\mathcal{S}$. Now, 
	\begin{equation}
		\rho = \frac{c - \sqrt{c^2 - 4D\beta}}{2D\beta} \leq \frac{2}{c},  
	\end{equation}
	for all $c \geq 2\sqrt{D\beta}$. Hence, \eqref{etaCrit2} becomes
	\begin{equation}
		\eta_\mathrm{crit} \leq \frac{c}{2} - \frac{\chi(V)}{c} \leq \frac{c}{2} + \frac{c}{4} < c,
	\end{equation} 
	making the surface $N(U) < cU$ for all relevant parameter values. This completes the proof. \qed
\end{proof}

\section*{Acknowledgements}

\noindent The author was supported by a PIMS PDF held at the University of Victoria.

\end{document}